\pdfoutput=1
\documentclass[10pt]{amsart}
\usepackage{amsmath,amssymb,latexsym,soul}
\usepackage{cite,mathrsfs,amsfonts}
\usepackage{color,enumitem,graphicx}
\usepackage[colorlinks=true,urlcolor=blue,
citecolor=red,linkcolor=blue,linktocpage,pdfpagelabels,
bookmarksnumbered,bookmarksopen]{hyperref}
\usepackage[english]{babel}

\usepackage[pdftex,a4paper,left=2.6cm,right=2.6cm,top=1.35cm,bottom=3.95cm]{geometry}

\usepackage[hyperpageref]{backref}

\usepackage{mathrsfs}

\usepackage{lmodern}

\numberwithin{equation}{section}

\newtheorem{theorem}{Theorem}[section]
\newtheorem{lemma}[theorem]{Lemma}

\newtheorem{proposition}[theorem]{Proposition}

\theoremstyle{definition}

\newcommand{\N}{{\mathbb N}}

\newcommand{\Z}{{\mathbb Z}}
\newcommand{\R}{{\mathbb R}}
\renewcommand{\S}{{\mathbb S}}
\newcommand{\eps}{\varepsilon}

\newcommand{\Hsnorm}[1]{\left\|#1\right\|}

\newcommand{\dHRN}{{\dot H^s(\R^N)}}
\newcommand{\onetwe}{1/20}
\newcommand{\M}{{\mathscr M}}
\newcommand{\NN}{{\mathscr N}}
\newcommand{\Y}{{\mathscr Y}_0}
\newcommand{\ZZ}{{\mathscr Z}_0}
\newcommand{\inp}[2]{\left\langle #1, #2\right\rangle}
\newcommand{\defi}{=}
\newcommand{\dist}{\mathrm{dist}}

\newcommand{\U}{U_{\varepsilon,z}}
\newcommand{\Uoz}{U_{1,0}}
\renewcommand{\u}{u_{\delta,\varepsilon,z}}
\newcommand{\uu}{u_{\bar{\varepsilon}^2,\bar{\varepsilon},z}}
\newcommand{\Ray}{\mathscr{R}}
\newcommand{\C}{\complement}

\title{Coron problem for fractional equations}

\author[S.\ Secchi]{Simone Secchi} \thanks{S.\ Secchi is supported by
  2012 FIRB \emph{Dispersive equations and Fourier analysis}.
  N.\ Shioji is partially supported by
  the Grant-in-Aid for Scientific Research (C) (No.\ 21540214 and No.\ 26400160)
  from Japan Society for the Promotion of Science.
  M.\ Squassina is supported by 2009 PRIN \emph{Variational and
    Topological Methods in the Study of Nonlinear Phenomena}}
\address{Dipartimento di Matematica e Applicazioni
\newline\indent 
Universit\`a di Milano Bicocca
\newline\indent
Edificio U5, Via Roberto Cozzi 55, 20125 Milano, Italy}
\email{simone.secchi@unimib.it}

\author[N.\ Shioji]{Naoki Shioji}
\address{Department of Mathematics
\newline\indent
Faculty of Engineering
\newline\indent
Yokohama National University
\newline\indent
Tokiwadai, Hodogaya-ku,
Yokohama 240-8501, Japan}
\email{shioji@ynu.ac.jp}

\author[M.\ Squassina]{Marco Squassina}
\address{Dipartimento di Informatica
\newline\indent
Universit\`a degli Studi di Verona
\newline\indent
C\'a Vignal 2, Strada Le Grazie 15, 37134 Verona, Italy}
\email{marco.squassina@univr.it}


\begin{document}

\begin{abstract}
 We prove that the critical problem for the fractional Laplacian in an annular type domain
 admits a positive solution provided that the inner hole is sufficiently small.
\end{abstract}

\maketitle

\section{Introduction}

\noindent
Let  $N\geq 3$ and $\Omega$ be a smooth bounded domain of $\R^N$. The classical formulation of Coron problem goes back to 1984 and
says that if there is a point
$x_0\in \R^N$ and radii $R_2>R_1>0$ such that
\begin{equation}
\label{domain-ass}
\{R_1\leq |x-x_0|\leq R_2\}\subset \Omega,\qquad
\{|x-x_0|\leq R_1\}\not\subset \overline\Omega,
\end{equation}
then the critical elliptic problem 
\begin{equation}
\begin{cases}
\label{probcrit-s1}
-\Delta u=u^{\frac{N+2}{N-2}} & \text{in $\Omega$,} \\
\,\, u>0 & \text{in $\Omega$},\\
\,\, u=0 & \text{on $\partial\Omega$,}
\end{cases}
\end{equation}
admits a solution provided that $R_2/R_1$ is 
sufficiently large \cite{coron}. A few years later Bahri and Coron \cite{bahricoron}, in a seminal paper, considerably improved this 
existence result by showing, via sofisticated
topological arguments based upon homology theory, that 
\eqref{probcrit-s1} admits a solution provided that
$H_m(\Omega,\Z_2)\not=\{0\}$ for some $m>0$. 
Furthermore, in \cite{dancer,ding,passaseo} the authors show that existence
of a solution is possible also in some contractible domains.
Let $N\geq 2$ and $s\in (0,1)$ with $N>2s$,
and consider the nonlocal fractional problem 
\begin{equation}
\begin{cases}
\label{probcrit}
(-\Delta)^s u=u^{\frac{N+2s}{N-2s}} & \text{in $\Omega$,} \\
\,\, u>0 & \text{in $\Omega$},\\
\,\, u=0 & \text{in $\R^N\setminus\Omega$},
\end{cases}
\end{equation}
involving the fractional Laplacian $(-\Delta)^s$.
Here, for smooth functions $\varphi$, $(-\Delta)^s\varphi$ is defined by
\[
(-\Delta)^s \varphi(x)
\defi
C(N,s)\lim_{\varepsilon\rightarrow 0}\int_{\C B_\varepsilon(x)}\frac{\varphi(x)-\varphi(y)}{|x-y|^{N+2s}}\,dy
,\qquad x\in\R^N,
\]
where
\begin{equation}
\label{cos}
C(N,s)\defi\left(\int_{\R^N}\frac{1-\cos \zeta_1}{|\zeta|^{N+2s}}\,d\zeta\right)^{-1};
\end{equation}
see \cite{DiNezza}.
Fractional Sobolev spaces are introduced in the middle part of the last
century, especially in the framework of harmonic analysis. 
More recently, after the paper of Caffarelli and Silvestre \cite{caffarelli},
a large amount of papers were written on problems
which involve the fractional diffusion $(-\Delta)^s$, $0<s<1$. Due to its nonlocal character, working on bounded
domains imposes that an appropriate variational formulation of the problem is to consider functions 
on $\R^N$ with the condition $u=0$ in $\R^N\setminus\Omega$ replacing the
boundary condition $u=0$ on $\partial\Omega$. We set $X_0=\{u\in \dot H^s(\R^N): \text{$u=0$ in $\R^N\setminus\Omega$}\}$
 and we consider the formulation
\begin{equation*}
\int_{\mathbb{R}^N}  (-\Delta)^{s/2} u (-\Delta)^{s/2}\varphi\,dx
= \int_{\Omega} u^{\frac{N+2s}{N-2s}}\varphi  \,dx
\qquad \text{for all $\varphi\in X_0$}.
\end{equation*}
It has been proved recently \cite[Corollary 1.3]{ros-oton} that if $\Omega$
is a star-shaped domain, then problem  \eqref{probcrit} does not admit
solutions and that for exponents larger that $(N+2s)/(N-2s)$ the problem
does not admit any nontrivial solution thus dropping the positivity
requirement.
It is then natural to think that, as in the local case $s=1$, by assuming suitable
geometrical or topological conditions on $\Omega$ one can get the existence
of nontrivial solutions.
We note that 
Capella~\cite{capella} studies the problem for the particular case
$s=1/2$ 
by using
the Caffarelli reduction to transform the problem in a local form
and that
Servadei and Valdinoci~\cite{valdserv-tams} studies
the Brezis-Nirenberg problem with the fractional Laplacian.
\vskip2pt
\noindent
The main result of the paper is the following Coron type result in the fractional setting.
\begin{theorem}
\label{mainthm}
If \eqref{domain-ass} holds,
then \eqref{probcrit} admits a weak solution in $X_0$ for $R_2/R_1$ 
sufficiently large.
\end{theorem}


\noindent
We roughly recall Coron's argument \cite{coron}
for the case $s=1$.
Although the corresponding
Rayleigh quotient does not attain the infimum value, say
$\S$, 
the global compactness theorem due to Struwe~\cite{struwe}
implies that
it satisfies the Palais-Smale condition at each level in
$(\S,2^{2/N}\S)$.
He introduced 
a test function defined on a small ball which contains
the small hole of $\Omega$,
and he showed that
under assumption \eqref{domain-ass},
the maximum value of the test function is less than $2^{2/N}\S$.
If there is no critical point
of the Rayleigh quotient in $(\S,2^{2/N}\S)$,
he showed that
the small ball can be retracted into its boundary,
which is a contradiction.
For the case $s\in (0,1)$,
one of the main difficulties that one has to face is to get 
a
uniform estimate for the
energy of truncations of the family of functions
\begin{equation}
\label{Talent}
U_{\eps,z}(x)\defi\Big(\frac{\varepsilon}{\varepsilon^2+|x-z|^2}\Big)^\frac{N-2s}{2},
\qquad z\in\R^N,\,\, \eps>0,
\end{equation}
which are precisely obtained in Propositions~\ref{one}-\ref{two}.
We note that $\U$ satisfies
$(-\Delta )^s u=u^{(N+2s)/(N-2s)}$
in $\R^N$ up to a constant,
and $\U$ with the constant factor
is called Talenti function for the fractional Laplacian.
The other difficultly for the case $s\in (0,1)$ is
global compactness.
We give a compactness result which is sufficient for our arguments.

\section{Preliminary results}
\noindent
We define
$$\dHRN\defi\{u \in L^\frac{2N}{N-2s}(\R^N): \|u\|<\infty\},$$
where
$$\|u\|\defi\left(
\iint_{\R^{2N}}\frac{|u(x)-u(y)|^2}{|x-y|^{N+2s}}\,dxdy\right)^\frac{1}{2}.$$
We also define
$$\inp{u}{v}\defi\iint_{\R^{2N}}
\frac{(u(x)-u(y))(v(x)-v(y))}{|x-y|^{N+2s}}\,dxdy
\quad\text{for each $u,v \in \dHRN$.}$$
Then we know that
$\dHRN$ is a Hilbert space with the inner product above,
it is continuously embedded into
$L^{2N/(N-2s)}(\R^N)$
and it holds
\begin{equation}
\label{careful}
\inp{u}{v}=\frac{2}{C(N,s)}
\int_{\mathbb{R}^N}  (-\Delta)^{s/2} u (-\Delta)^{s/2}v\,dx
\quad\text{for each $u,v \in \dHRN$,}
\end{equation}
where $C(N,s)$ is the constant given in \eqref{cos};
see \cite{DiNezza}.
We set
$$X_0\defi\{u\in \dot H^s(\R^N): \text{$u=0$ in
$\R^N\setminus\Omega$}\}.$$
Since it 
is a closed subspace of $\dHRN$,
$X_0$ itself is also a Hilbert space,
and we use the same symbols $\inp{\cdot}{\cdot}$ and $\|\cdot\|$
for its inner product and norm.
We note
\[
\|u\|=\left(\iint_{Q}\frac{|u(x)-u(y)|^2}{|x-y|^{N+2s}}dxdy\right)^{1/2}
\quad\text{for each $u \in X_0$,}
\]
where $Q\defi\R^{2N}\setminus (\complement \Omega\times \complement
\Omega)$.

\noindent
Since we have \eqref{careful},
for the sake of simplicity,
we will find a positive weak solution of
\begin{equation}
\label{problem}
\begin{cases}
(-\Delta)^s u=\frac{C(N,s)}{2}|u|^{\frac{4s}{N-2s}}u & \text{in $\Omega$,} \\
\,\, u=0 & \text{in $\R^N\setminus\Omega$,}
\end{cases}
\end{equation}
which is equivalent to find a weak solution of \eqref{probcrit}.
Here,
we say 
$u\in X_0$ is a weak solution to \eqref{problem}
if it satisfies
\begin{equation*}
\iint_Q\frac{(u(x)-u(y))(\varphi(x)-\varphi(y))}{|x-y|^{N+2s}}\,dxdy
=\int_\Omega |u|^{\frac{4s}{N-2s}}u\varphi\,dx
\qquad \text{for each $\varphi \in X_0$.}
\end{equation*}


\vskip1pt
\noindent
Without loss of generality,
we may assume \eqref{domain-ass} with $x_0=0 \not\in
\overline{\Omega}$,
$R_2$ is fixed with $R_2>10$ and $R_1=\delta
\in (0,\onetwe]$ which will be fixed later.
We set $B_r\defi\{x\in \R^N: |x|\leq r\}$ for $r>0$.
Without loss of generality,
we may also assume
$\Omega\cap B_\delta=\emptyset$ and
$B_5\setminus B_{3\delta/2}\subset \Omega$.
Let $\varphi_\delta : \R^N\rightarrow [0,1]$ be  
a smooth radially symmetric function
such that
\[
    \varphi_\delta(x)=
\begin{cases}
0 & \text{if $0 \leq |x|\leq 2\delta$ and $|x|\geq 4$,}\\
1 & \text{if $4\delta\leq |x|\leq 3$,}
\end{cases}
\]
\[
|\nabla \varphi_\delta(x)|\leq \delta^{-1}, 
\quad\text{for $x\in B_{4\delta}$,}
\qquad
|\nabla \varphi_\delta(x)|\leq 2, \quad\text{for $x\in \C B_3$.}
\]
For $\delta$, $\varepsilon\in(0,\onetwe]$ and $z\in B_1$, 
we set
\[
\u(x)\defi\varphi_\delta(x)\U(x),
\]
where the $\U$ were defined in \eqref{Talent}.
The next estimates will be crucial for the proof of Theorem~\ref{mainthm}.
\begin{proposition}
\label{one}
There exists $C_1>0$
such that
\begin{equation} \label{stima1}
\Hsnorm{\u}^2
\leq
\Hsnorm{\U}^2
+
C_1\left(
 \left(\frac{\delta}{\varepsilon}\right)^{N-2s}
+
\left(\frac{\delta}{\varepsilon}\right)^{N+2-2s}+\varepsilon^{N-2s}\right)
\end{equation}
for each 
$\delta$, $\varepsilon\in (0,\onetwe]$
and $z \in B_1$,
and
\begin{equation} \label{stima2}
\Hsnorm{\u}^2
\leq
\Hsnorm{\U}^2
+C_1\varepsilon^{N-2s}(1+\delta^{-2s})
\end{equation}
for each 
$\delta$, $\varepsilon\in (0,\onetwe]$
and $z \in B_1\setminus B_{1/2}$.
\end{proposition}
\begin{proof}
Let $\delta$, $\varepsilon\in (0,\onetwe]$
and $z \in B_1$.
We define
\begin{align*}
D&=\{(x,y)\in (B_4\times\C B_3)\cup(\C B_3\times B_4) : |x-y|>1\},\\
E&=\{(x,y)\in (B_4\times\C B_3)\cup(\C B_3\times B_4): |x-y|\leq 1\},\\
\widetilde{D}&=\{(x,y)\in 
(B_{4\delta}\times (B_4\setminus B_{4\delta}))
\cup
((B_4\setminus B_{4\delta})\times B_{4\delta})
: |x-y|>\delta\}
\end{align*}
and
\[
\widetilde{E}=
(B_{4\delta}\times B_{4\delta})
\cup
\left\{(x,y)\in 
(B_{4\delta}\times (B_4\setminus B_{4\delta}))
\cup
((B_4\setminus B_{4\delta})\times B_{4\delta})
: |x-y|\leq \delta \right\}.
\]
Then we have
\[
\R^{2N}=\widetilde{E}\cup \widetilde{D}\cup E\cup D
\cup ((B_3\setminus B_{4\delta})\times (B_3\setminus B_{4\delta}))
\cup (\C B_4 \times \C B_4).
\]
We remark that this is \emph{not} a disjoint union.
We can easily see that
\[
\int\limits_{(B_3\setminus B_{4\delta})\times (B_3\setminus B_{4\delta})}
\left(\frac{|\u(x)-\u(y)|^2}{|x-y|^{N+2s}}
-\frac{|\U(x)-\U(y)|^2}{|x-y|^{N+2s}}\right)\,dxdy=0
\]
and
\[
\int\limits_{\C B_4 \times \C B_4}\left(\frac{|\u(x)-\u(y)|^2}{|x-y|^{N+2s}}
-\frac{|\U(x)-\U(y)|^2}{|x-y|^{N+2s}}\right)\,dxdy\leq 0.
\]
We shall denote by $C$ generic positive constants, possibly varying from line to line, and which do not depend on 
$\delta, \varepsilon\in (0,\onetwe]$ and $z\in B_1$.
For each $(x,y)\in \R^{2N}$, we have
\begin{multline}
\label{calc}
\frac{|\u(x)-\u(y)|^2}{|x-y|^{N+2s}}
-\frac{|\U(x)-\U(y)|^2}{|x-y|^{N+2s}}\\
=\frac{(\u(x)+\U(x)-\u(y)-\U(y)) 
 (\u(x)-\U(x)+\U(y)-\u(y))}{|x-y|^{N+2s}} \\
\leq 
\frac{2\U(x)\U(y)}{|x-y|^{N+2s}}.
\end{multline}
From $z \in B_1$, we have
\begin{multline*}
\int\limits_{B_4\times \C B_3,\,|x-y|>1} \left(\frac{|\u(x)-\u(y)|^2}{|x-y|^{N+2s}}
-\frac{|\U(x)-\U(y)|^2}{|x-y|^{N+2s}}\right)\,dxdy\\
 \qquad\leq
\int\limits_{B_4\times \C B_3,\,|x-y|>1}
\frac{2\U(x)\U(y)}
{|x-y|^{N+2s}}\,dxdy
\leq  C\varepsilon^\frac{N-2s}{2}
\int\limits_{B_4\times \C B_3,\,|x-y|>1}
\frac{\left(\frac{\varepsilon}{\varepsilon^2+|x-z|^2}\right)^\frac{N-2s}{2}}
{|x-y|^{N+2s}}\,dxdy\\
\leq
C\varepsilon^{N-2s}
\int\limits_{|\xi|\leq 5}\frac{d\xi}{(\varepsilon^2+|\xi|^2)^\frac{N-2s}{2}} 
\int\limits_{|\eta|> 1}\frac{d\eta}{|\eta|^{N+2s}}
=
C\varepsilon^{N-2s}\* \varepsilon^{2s}
\int\limits_{|\xi|\leq 5/\varepsilon}\frac{d\xi}{(1+|\xi|^2)^\frac{N-2s}{2}} \leq 
C\varepsilon^{N-2s}.
\end{multline*}%
So we can infer
\[
\int_{D} \left(\frac{|\u(x)-\u(y)|^2}{|x-y|^{N+2s}}
-\frac{|\U(x)-\U(y)|^2}{|x-y|^{N+2s}}\right)\,dxdy\leq C\varepsilon^{N-2s}.
\]
We note
\[
\nabla\U(x)=-(N-2s)\left(\frac{\varepsilon}{\varepsilon^2+|x-z|^2}
\right)^\frac{N-2s}{2}\frac{x-z}{\varepsilon^2+|x-z|^2},
\]
and
\[ 
\frac{|x-z|}{\varepsilon^2+|x-z|^2}\leq \frac{1}{2\varepsilon}.
\]
Since $|\nabla\varphi_\delta(x)|\leq 2$ for $|x|\geq
 4\delta$,
$z\in B_1$ and $|tx+(1-t)y|\geq 2$ for each $(x,y)\in E$ and $t\in [0,1]$,
\begin{multline*}
\int_{E} \left(\frac{|\u(x)-\u(y)|^2}{|x-y|^{N+2s}}
-\frac{|\U(x)-\U(y)|^2}{|x-y|^{N+2s}}\right)\,dxdy\\
 \leq
\int_{E}
\frac{|\u(x)-\u(y)|^2}{|x-y|^{N+2s}}\,dxdy=\int_{E}
\frac{|\int_0^1(\nabla\u)(tx+(1-t)y)\cdot
 (x-y)\,dt|^2}{|x-y|^{N+2s}}\,dxdy\\
\qquad \leq\int_{E}
\frac{\int_0^1(8|\U(tx+(1-t)y)|^2
+2|(\nabla\U)(tx+(1-t)y)|^2)\,dt}{|x-y|^{N+2s-2}}\,dxdy\\
\qquad
\leq
C\varepsilon^{N-2s}
\int_{E}
\frac{dxdy}{|x-y|^{N+2s-2}}\\
\qquad
\leq
C\varepsilon^{N-2s}
\int_{|\xi|\leq 4}\,d\xi 
\int_{|\eta|\leq 1}\frac{d\eta}{|\eta|^{N+2s-2}}
=C\varepsilon^{N-2s}.
\end{multline*}
From \eqref{calc}, we also have
\begin{multline*}
\int_{\widetilde{D}} \left(\frac{|\u(x)-\u(y)|^2}{|x-y|^{N+2s}}
-\frac{|\U(x)-\U(y)|^2}{|x-y|^{N+2s}}\right)\,dxdy\\
 \leq
2\int_{\widetilde{D}}
\frac{\U(x)\U(y)}
{|x-y|^{N+2s}}\,dxdy
=\frac{2}{\varepsilon^{N-2s}}\int_{\widetilde{D}}
\frac{dxdy}
{|x-y|^{N+2s}}
\\
\qquad
\leq
\frac{C}{\varepsilon^{N-2s}}\int_{|\xi|\leq 4\delta}\,d\xi 
\int_{|\eta|> \delta}\frac{d\eta}{|\eta|^{N+2s}}
\leq C\left(\frac{\delta}{\varepsilon}\right)^{N-2s}.
\end{multline*}
Since 
$|\nabla \varphi_\delta(x)|\leq 1/\delta$ for $x\in B_{4\delta}$,
we have 
{\allowdisplaybreaks
\begin{align*}
&\int_{\widetilde{E}} \left(\frac{|\u(x)-\u(y)|^2}{|x-y|^{N+2s}}
-\frac{|\U(x)-\U(y)|^2}{|x-y|^{N+2s}}\right)\,dxdy \\
&\qquad\leq
\int_{\widetilde{E}}
\frac{|\u(x)-\u(y)|^2}{|x-y|^{N+2s}}\,dxdy\\
&\qquad =\int_{\widetilde{E}}
\frac{|\int_0^1(\nabla\u)(tx+(1-t)y)\cdot (x-y)\,dt|^2}{|x-y|^{N+2s}}\,dxdy\\
&\qquad \leq\int_{\widetilde{E}}
\frac{\int_0^1((1/\delta)^2|\U(tx+(1-t)y)|^2
+|(\nabla\U)(tx+(1-t)y)|^2)\,dt}{|x-y|^{N+2s-2}}\,dxdy\\
&\qquad
\leq C\left(\frac{1}{\delta^2\varepsilon^{N-2s}}
+\frac{1}{\varepsilon^{N-2s}}\cdot\frac{1}{\varepsilon^{2}}
\right)
\int_{\widetilde{E}}
\frac{dxdy}{|x-y|^{N+2s-2}}\\
&\qquad\leq
C\left(\frac{1}{\delta^2\varepsilon^{N-2s}}+\frac{1}{\varepsilon^{N+2-2s}}
\right)
\int_{|\xi|\leq 4\delta}\,d\xi 
\int_{|\eta|\leq \delta}\frac{d\eta}{|\eta|^{N+2s-2}}\\
&\qquad=C\left(\left(\frac{\delta}{\varepsilon}\right)^{N-2s}
+\left(\frac{\delta}{\varepsilon}\right)^{N+2-2s}
\right).
\end{align*}
}%
By the inequalities above, we obtain \eqref{stima1}.
Let $z \in B_1\setminus B_{1/2}$.
In order to obtain \eqref{stima2}
we need to consider the integrals on $\widetilde{D}$ and
 $\widetilde{E}$.
We have
\begin{multline*}
\int\limits_{(B_4\setminus B_{4\delta})\times  B_{4\delta},\,|x-y|>\delta} 
\left(\frac{|\u(x)-\u(y)|^2}{|x-y|^{N+2s}}
-\frac{|\U(x)-\U(y)|^2}{|x-y|^{N+2s}}\right)\,dxdy\\
 \leq
\int\limits_{(B_4\setminus B_{4\delta})\times  B_{4\delta},\,|x-y|>\delta} 
\frac{2\U(x)\U(y)}
{|x-y|^{N+2s}}\,dxdy
\leq  C\varepsilon^\frac{N-2s}{2}
\int\limits_{(B_4\setminus B_{4\delta})\times B_{4\delta},\,|x-y|>\delta} 
\frac{\big(\frac{\varepsilon}{\varepsilon^2+|x-z|^2}\big)^\frac{N-2s}{2}}
{|x-y|^{N+2s}}\,dxdy\\
\leq
C\varepsilon^{N-2s}
\int_{|\xi|\leq 5}\frac{d\xi}{(\varepsilon^2+|\xi|^2)^\frac{N-2s}{2}} 
\int_{|\eta|> \delta}\frac{d\eta}{|\eta|^{N+2s}}\\
\qquad
=
C\varepsilon^{N-2s}\delta^{-2s}\cdot \varepsilon^{2s}
\int_{|\xi|\leq 5/\varepsilon}\frac{d\xi}{(1+|\xi|^2)^\frac{N-2s}{2}} 
= C\varepsilon^{N-2s}\delta^{-2s}.
\end{multline*}
Hence
\[
\int_{\widetilde{D}} 
\left(\frac{|\u(x)-\u(y)|^2}{|x-y|^{N+2s}}
-\frac{|\U(x)-\U(y)|^2}{|x-y|^{N+2s}}\right)\,dxdy
\leq C\varepsilon^{N-2s}\delta^{-2s}.
\]
Since $|\nabla\varphi_\delta(x)|\leq 1/\delta$ for $x\in \R^N$,
$z\in B_1 \setminus B_{1/2}$ 
and $|tx+(1-t)y|\leq 5\delta\leq 1/4$ for each $(x,y)\in \widetilde{E}$ and $t\in [0,1]$,
we have
{\allowdisplaybreaks
\begin{align*}
&\int_{\widetilde{E}} \left(\frac{|\u(x)-\u(y)|^2}{|x-y|^{N+2s}}
-\frac{|\U(x)-\U(y)|^2}{|x-y|^{N+2s}}\right)\,dxdy\\
& \qquad\leq
\int_{\widetilde{E}}
\frac{|\u(x)-\u(y)|^2}{|x-y|^{N+2s}}\,dxdy\\
& \qquad=\int_{\widetilde{E}}
\frac{|\int_0^1(\nabla\u)(tx+(1-t)y)\cdot
 (x-y)\,dt|^2}{|x-y|^{N+2s}}\,dxdy\\
&\qquad \leq\int_{\widetilde{E}}
\frac{\int_0^1((1/\delta)^2|\U(tx+(1-t)y)|^2
+|(\nabla\U)(tx+(1-t)y)|^2)\,dt}{|x-y|^{N+2s-2}}\,dxdy\\
&\qquad
\leq
C\varepsilon^{N-2s}\delta^{-2}
\int_{\widetilde{E}}
\frac{dxdy}{|x-y|^{N+2s-2}}\\
&\qquad
\leq
C\varepsilon^{N-2s}\delta^{-2}
\int_{|\xi|\leq 4\delta}\,d\xi 
\int_{|\eta|\leq \delta}\frac{d\eta}{|\eta|^{N+2s-2}}
=C(\varepsilon\delta)^{N-2s}
\leq C\varepsilon^{N-2s}.
\end{align*}
}%
Thus,
we obtain the second desired inequality.
\end{proof}
\begin{proposition}
\label{two}
There exists $C_2>0$
such that
\begin{equation}\label{stima3}
\int_{\R^N} |\u|^\frac{2N}{N-2s}\,dx\geq
\int_{\R^N} |\U|^\frac{2N}{N-2s}\,dx-C_2
\Big(
\Big(\frac{\delta}{\varepsilon}\Big)^N
+
\varepsilon^N\Big)
\end{equation}
for each 
$\delta$, $\varepsilon\in (0,\onetwe]$
and $z \in B_1$,
and
\begin{equation} \label{stima4}
\int_{\R^N} |\u|^\frac{2N}{N-2s}\,dx\geq
\int_{\R^N} |\U|^\frac{2N}{N-2s}\,dx-C_2\varepsilon^N
\end{equation}
for each 
$\delta$, $\varepsilon\in (0,\onetwe]$
and $z \in B_1\setminus B_{1/2}$.
\end{proposition}
\begin{proof}
Let $\delta$, $\varepsilon\in (0,\onetwe]$ and $z \in B_1$.
We have
\begin{multline*}
\int_{\R^N} |\U|^\frac{2N}{N-2s}\,dx
-\int_{\R^N} |\u|^\frac{2N}{N-2s}\,dx\\
\leq
\int_{|x|\leq
 4\delta}\left(\frac{\varepsilon}{\varepsilon^2+|x-z|^2}\right)^N
\,dx
+
\int_{|x|\geq
3}\left(\frac{\varepsilon}{\varepsilon^2+|x-z|^2}\right)^N
\,dx 
\leq C\left(\frac{\delta}{\varepsilon}\right)^N+C\varepsilon^N,
\end{multline*}
which yields \eqref{stima3}.
By a similar calculation,
we can obtain \eqref{stima4} as well.
\end{proof}

\noindent
Let $I: \dot H^s(\R^N)\to\R$ be given by
$$
I(u)\defi\frac{1}{2}\Hsnorm{u}^2-\frac{N-2s}{2N}\int_{\R^N} |u|^{2N/(N-2s)}
\,dx
\quad\text{for $u \in \dHRN$,}
$$
and let
$I_0: X_0\rightarrow\R$ be its restriction to $X_0$, i.e.,
$$I_0(u)=I(u)\quad\text{for $u \in X_0$.}$$
\noindent
Next, let us define $\Ray \colon \dot H^s(\R^N) \setminus \{0\} \to \mathbb{R}$ by
\begin{equation*}
\Ray(u)\defi
\frac{\Hsnorm{u}^2}{\NN(u)},
\end{equation*}
where
\begin{equation}\label{N}
\NN(u) \defi \Big( \int_{\R^N} |u|^{\frac{2N}{N-2s}}\, dx \Big)^{\frac{N-2s}{N}}.
\end{equation}
\noindent
We also define $\NN_0$ and $\Ray_0$
by the restrictions of $\NN$ and $\Ray$
to $X_0\setminus\{0\}$, respectively.
That is,
$$\NN_0(u)=\NN(u)\quad\text{and}\quad
\Ray_0(u)=\Ray(u)\quad\text{for $u\in X_0\setminus\{0\}$.}$$

\noindent
\begin{lemma}
\label{link-prop}
$\Ray_0 \in C^1(X_0\setminus\{0\})$, 
and if $\Ray_0'(v)=0$ with $v \in X_0$,
then
$I_0'(\lambda v)=0$ with some $\lambda>0$.
\end{lemma}
\begin{proof}
We can easily see $\Ray_0\in C^1(X_0\setminus\{0\})$.
Let $v \in X_0$.
Since we have
$$
\Ray_0'(v)(\varphi)
= \frac{2\NN_0(v) \iint_Q \frac{(v(x)-v(y))(\varphi(x)-\varphi(y))}{|x-y|^{N+2s}}\, dxdy - 2 \|v\|^2 \NN_0(v)^{-\frac{2s}{N-2s}} \int_\Omega |v|^{\frac{4s}{N-2s}}v \varphi \, dx}{\NN_0(v)^2}
$$
for every $\varphi\in X_0$,
we have
$\Ray_0'(v) =0$ if and only if, for every $\varphi \in X_0$,
\[
\iint_Q \frac{(v(x)-v(y))(\varphi(x)-\varphi(y))}{|x-y|^{N+2s}}\, dxdy  =  \frac{\|v\|^2}{\int_\Omega |v|^{\frac{2N}{N-2s}}\, dx} \int_\Omega |v|^{\frac{4s}{N-2s}}v\varphi \, dx.
\]
Setting $\lambda$ by
\begin{equation}
\label{lambda-def}
\lambda^{\frac{4s}{N-2s}} = \frac{\|v\|^2}{\int_\Omega |v|^{\frac{2N}{N-2s}}\, dx},
\end{equation}
we have $I_0(\lambda v)=0$.
This concludes the proof.
%
\end{proof}

\noindent
We define a manifold of codimension one by setting
\begin{equation}
\M\defi \left\{ u \in X_0 : \int_\Omega |u|^{\frac{2N}{N-2s}}\,dx =1 \right\}.
\end{equation}

\begin{lemma}
\label{Ps-connect}
Let $\{v_n\}_n\subset \M$ be a Palais-Smale sequence for $\Ray_0$ 
at level $c$.
Then
\[
u_n=\lambda_n v_n,\qquad
\lambda_n\defi \Ray_0(v_n)^{(N-2s)/(4s)}
\]
is a Palais-Smale sequence for $I_0$ at level $(s/N) c^{N/(2s)}$.
\end{lemma}
\begin{proof}
By following the computations of Lemma~\ref{link-prop}, if $\lambda_n$
is defined as in \eqref{lambda-def}, we have
\[
\frac{1}{2}\Ray_0'(v_n)(\varphi)
=\iint_Q \frac{(v_n(x)-v_n(y))(\varphi(x)-\varphi(y))}{|x-y|^{N+2s}}dxdy  
- \lambda_n^{\frac{4s}{N-2s}}\int_\Omega |v_n|^{\frac{4s}{N-2s}}v_n\varphi \, dx
\]
for every $\varphi\in X_0$. Hence, in turn, by
multiplying this identity by $\lambda_n$, we conclude that
\[
I_0'(u_n)(\varphi)
=\iint_Q \frac{(u_n(x)-u_n(y))(\varphi(x)-\varphi(y))}{|x-y|^{N+2s}}dxdy  
- \int_\Omega |u_n|^{\frac{4s}{N-2s}}u_n\varphi \, dx
\]
for every $\varphi\in X_0$.
Recalling \eqref{N} and \eqref{lambda-def}, we have 
\begin{equation*}
\lambda_n = \|v_n\|^{\frac{N-2s}{2s}} = \Ray_0(v_n)^{\frac{N-2s}{4s}} .
\end{equation*}
From $\Ray_0(v_n)=c+o(1)$ and $\{v_n\}_n\subset \M$, 
$\{v_n\}_n$ is bounded in $X_0$ and so is $\{\lambda_n\}_n$.
%
In particular, it follows that $I_0'(u_n)\to 0$ in $X_0'$ as $n\to\infty$.
Moreover, $\{u_n\}_n$ is bounded in $X_0$ as well, yielding
$$
o(1)=I_0'(u_n)(u_n)=\|u_n\|^2-\int_\Omega |u_n|^\frac{2N}{N-2s}\,dx.
$$ 
These facts imply that
\begin{equation*}
\lim_{n \to \infty} I_0(u_n)=\frac{s}{N}\lim_{n \to \infty} 
\int_\Omega |u_n|^\frac{2N}{N-2s}\,dx
=\frac{s}{N}\lim_{n \to \infty} \lambda_n^{\frac{2N}{N-2s}}  
=\frac{s}{N}\big (\lim_{n \to \infty} \Ray_0(v_n)^{\frac{N-2s}{4s}}\big)^{\frac{2N}{N-2s}}=\frac{s}{N} c^{N/(2s)},
\end{equation*}
concluding the proof.
\end{proof}

\noindent
Let us set
\begin{equation*}
\label{S1}
\S\defi\inf\{\Ray(u): u \in \dHRN\setminus\{0\}\}.
\end{equation*}
\noindent
By \cite{bestconst},
we know that
$$\Ray(\U)=\S
\quad\text{for each $\varepsilon>0$ and $z \in \R^N$,}$$
only these functions with any nonzero constant factor attain the infimum,
\begin{equation*}
\label{S2}
\S=\inf\{\Ray_0(u): u \in X_0\setminus\{0\}\},
\end{equation*}
and the infimum is never attained in the latter case.
We also have the following result for sign-changing weak solutions.
\begin{lemma}
\label{segno-lem}
Let $u\in X_0$ be a sign-changing weak solution to \eqref{problem},
then  $\|u\|^2\geq 2\S^{N/(2s)}$.
Moreover, the same conclusion 
holds for sign-changing critical points of $I$.
\end{lemma}
\begin{proof}
We have 
$u^\pm\in X_0\setminus\{0\}$
and
$$
|u(x)-u(y)|^2=|u^+(x)-u^+(y)|^2+|u^-(x)-u^-(y)|^2+2u^+(y)u^-(x)+2u^+(x)u^-(y)
$$
for every $x,y\in\R^N$, where $u^-(x)=-\min\{u(x),0\}$. This, in turn, implies
$$
\|u\|^2=\|u^+\|^2+\|u^-\|^2+4\iint_Q \frac{u^+(y)u^-(x)}{|x-y|^{N+2s}}\,dxdy.
$$
By multiplying equation \eqref{problem} by $u^\pm$ easily yields
\begin{align*}
\|u^+\|^2+2\iint_Q \frac{u^+(y)u^-(x)}{|x-y|^{N+2s}} \,dxdy
&=\int_\Omega |u^+|^\frac{2N}{N-2s}\,dx, 
\\
\|u^-\|^2+2\iint_Q \frac{u^+(y)u^-(x)}{|x-y|^{N+2s}} \,dxdy
&=\int_\Omega |u^-|^\frac{2N}{N-2s}\,dx.  
\end{align*}
Combining these equalities with
 $\S\|u^\pm\|_{L^{2N/(N-2s)}}^2\leq \|u^\pm\|^2$, yields
$\int_\Omega |u^\pm|^{2N/(N-2s)}\geq\S^{N/(2s)}$, concluding the proof.
\end{proof}

\noindent
Now, we show the following compactness result.
In order to show it,
we follow the arguments in \cite[Section~8.3]{willem},
which treat the case $s=1$.
\begin{proposition}
\label{s-struwe}
Let $\{u_n\}_n\subset X_0$ be a Palais-Smale sequence for $I_0$ at
 level $c$ with
$$
\frac{s}{N} \S^{N/(2s)}\leq
c<\frac{2s}{N} \S^{N/(2s)}.
$$
If $(s/N)\S^{N/(2s)}<c<(2s/N) \S^{N/(2s)}$,
then
$\{u_n\}_n$
converges strongly to 
a nontrivial constant-sign weak solution 
to problem \eqref{problem}
up to a subsequence, 
and if $c=(s/N)\S^{N/(2s)}$,
then
there exist a nontrivial constant-sign weak solution 
$v\in\dot H^s(\R^N)$ to problem
\begin{equation}
\label{probcrit2-bis}
(-\Delta)^s v=\frac{C(N,s)}{2}|v|^{\frac{4s}{N-2s}}v\quad \text{in $\R^N$},
\end{equation}
$\{x_n\}_n\subset\Omega$ and 
$\{r_n\}_n\subset (0,\infty)$ 
with $r_n\rightarrow 0$
such that
$\{u_n-r_n^{(2s-N)/2}
v((\cdot-x_n)/r_n)\}_n$
converges strongly to $0$ in $\dHRN$
up to a subsequence.
\end{proposition}
\begin{proof}
First, we note that
$\{u_n\}_n$ is bounded and
$I_0(u_n)=(s/N)\|u_n\|^2+o(1)$.
We may assume that $\{u_n\}_n$ converges weakly to $u$ in $X_0$.
Then
$u$ is a possibly trivial solution to \eqref{problem}
and
$$\|u\|^2 \leq\varliminf_{n\rightarrow\infty}\|u_n\|^2
=\frac{N}{s}\lim_{n\rightarrow \infty}I_0(u_n)< 2\S^{N/(2s)}.$$
From Lemma~\ref{segno-lem},
$u$ is not sign-changing.
By a similar argument as in \cite[Lemma~8.10]{willem},
we have
$$I'(u_n-u)\rightarrow 0, \quad
I(u_n-u)\rightarrow c-I_0(u)\quad\text{and}\quad \|u_n-u\|^2\rightarrow 
\frac{Nc}{s}-\|u\|^2.$$
If $\|u_n-u\|_{L^{2N/(N-2s)}}\rightarrow 0$,
we can infer that
$\|u_n-u\|\rightarrow 0$,
$(s/N)\S^{N/(2s)}< c< (2s/N)\S^{N/(2s)}$
and $u$ is a nontrivial constant-sign solution to \eqref{problem}.
From here,
we consider the case $\|u_n-u\|_{L^{2N/(N-2s)}}\not\rightarrow 0$.
Taking small $\delta>0$,
we may assume that $\int_{\R^N}|u_n-u|^{2N/(N-2s)}dx\geq \delta,$
for each $n \in \N$.
As in the proof of \cite[ 2) and 3) of Theorem~8.13]{willem},
we can choose appropriate sequences $\{x_n\}_n\subset \Omega$
and $\{r_n\}_n \subset(0,\infty)$
such that
the sequence $\{v_n\}_n\subset \dHRN$
defined by
$$
v_n(x)=r_n^{(N-2s)/2}(u_n-u)(r_n x+x_n)
$$
converges weakly to $v \in \dHRN\setminus\{0\}$.
We have
\begin{equation}
\label{limit}
\|v\|^2\leq \varliminf_{n\rightarrow\infty}\|v_n\|^2
=\lim_{n\rightarrow\infty}\|u_n-u\|^2
=\frac{Nc}{s}-\|u\|^2<2\S^{N/(2s)}-\|u\|^2.
\end{equation}
By the boundedness of $\Omega$ and $v \neq 0$,
we may assume
$r_n\rightarrow 0$ and $x_n\rightarrow x_0 \in \overline{\Omega}$.
We may also assume that
$\{\dist(x_n,\partial\Omega)/r_n\}_n$
has a limit value in $[0,\infty]$.
Assume that this limit value is finite.
Then $v$ is a solution to the problem
$$(-\Delta)^s v=\frac{C(N,s)}{2}|v|^{\frac{4s}{N-2s}}v$$
in a half-space.
From \cite[Theorem~1.1 and Remark~4.2]{dCKP},
$v$ is locally bounded (although the boundedness of a domain is
assumed in \cite{dCKP}, the proof works for our case).
Then, in light of \cite[Corollary~3]{CFY} (see also \cite[Corollary~1.6]{fallweth}),
we know that
the above problem in any half-space
does not admit a nontrivial constant-sign solution.
So $v$ must be sign-changing,
but then
by a similar proof of Lemma~\ref{segno-lem},
we have $\|v\|^2\geq 2\S^{N/(2s)}$,
which contradicts \eqref{limit}.
So we find that
$\dist(x_n,\partial\Omega)/r_n\rightarrow\infty$.
Then we can see that $v$ is a nontrivial solution of
\eqref{probcrit2-bis}.
Using \eqref{limit} again,
we find that $v$ is constant-sign
and $u$ is trivial.
Setting 
$$
w_n(x)=u_n(x)-r_n^{(2s-N)/2}v((x-x_n)/r_n),
$$
we have
$$I'(w_n)\rightarrow 0, \quad
I(w_n)\rightarrow c-I(v)\quad\text{and}\quad \|w_n\|^2\rightarrow 
\frac{Nc}{s}-\S^{N/(2s)}<\S^{N/(2s)}.$$
If $\|w_n\|_{L^{2N/(N-2s)}}\not\rightarrow 0$,
repeating the argument above,
we can obtain a contradiction.
Hence,
we have
$\|w_n\|\rightarrow 0$ and 
$c=(s/N)\S^{N/(2s)}$.
Therefore,
we have shown our assertion.
\end{proof}


\noindent
We define $\Y, \ZZ: X_0\setminus\{0\}\rightarrow X_0$ by
$$\Y(u)=\frac{\nabla\NN_0(u)}{\|\nabla \NN_0(u)\|},
\qquad\ZZ(u)=\nabla \Ray_0(u)-\inp{\nabla \Ray_0(u)}{\Y(u)}\Y(u)
\quad\text{for each $u \in X_0\setminus\{0\}$.}$$
Here,
$\nabla \NN_0(u)$ and $\nabla \Ray_0(u)$ are the elements of $X_0$
respectively obtained from $\NN_0'(u)$ and $\Ray_0'(u)$
by the Riesz representation theorem.
We note that 
\begin{equation}
\label{flow}
\inp{\ZZ(u)}{\nabla \NN_0(u)}=0
\quad\text{and}\quad
\inp{\ZZ(u)}{\nabla \Ray_0(u)}
=\|\ZZ(u)\|^2
\quad\text{for each $u \in \M$.}
\end{equation}
The next proposition essentially says that
$\Ray_0|_\M$ 
satisfies the Palais-Smale condition 
at any level in $(\S,2^{2s/N}\S)$.
In the last section,
we give a negative gradient flow of $\Ray_0|_\M$;
see \eqref{flow3}.
\begin{proposition} \label{prop:PS}
Let $\{v_n\}_n\subset \M$
which satisfies $\ZZ(v_n)\rightarrow 0$ in $X_0$
and $\Ray_0(v_n)\rightarrow c\in (\S,2^{2s/N}\S)$.
Then $\{v_n\}_n$ has a convergent subsequence.
\end{proposition}
\begin{proof}
For each $u \in \M$,
we have
\begin{equation}
\label{flow2}
\begin{aligned}
 \|\ZZ(u)\|^2&=\|\nabla \Ray_0(u)-\inp{\nabla \Ray_0(u)}{\Y(u)}\Y(u)\|^2
=\|\Ray_0(u)\|^2-\inp{\Ray_0(u)}{\Y(u)}^2\\
& \geq\|\Ray_0(u)\|^2\frac{\inp{\Y(u)}{u}^2}{\|u\|^2}
 =\frac{2\|\Ray_0(u)\|^2}{\|u\|^2\|\nabla \NN_0(u)\|^2}.
\end{aligned}
\end{equation}
From our assumptions,
we can infer that $\nabla \Ray_0(v_n)\rightarrow 0$.
By virtue of Lemma~\ref{Ps-connect}, the sequence $u_n=\lambda_n v_n$, where $\lambda_n$
is defined as in \eqref{lambda-def}, is a Palais-Smale
sequence for $I_0$ at level $(s/N)c^{N/(2s)}$.
According to Proposition~\ref{s-struwe},
there exists a subsequence of $\{u_n\}_n$ which 
converges strongly in $X_0$.
Since we have $\lambda_n\rightarrow c^{(N-2s)/(4s)}$
from Lemma~\ref{Ps-connect},
we can see that our assertion holds.
\end{proof}
\noindent
For the reader's convenience, we give the following lemma.
\begin{lemma}
\label{Ekeland}
Let $\eta >0$ and $u \in \M$ with $\Ray_0(u)\leq \S+\eta$.
Then there exists $v \in \M$ such that
$\|u-v\|\leq \sqrt{\eta}$,
$\Ray_0(v)\leq \Ray_0(u)$
and
$\|\Ray_0'(v)\|\leq \sqrt{\eta}(1+1/\sqrt{\S})$.
\end{lemma}
\begin{proof}
By Ekeland's variational principle,
we can find $v\in \M$ such that
$\|u-v\|\leq \sqrt{\eta}$,
$\Ray_0(v)\leq \Ray_0(u)$
and
$\Ray_0(w)\geq \Ray_0(v)-\sqrt{\eta}\|w-v\|$
for each $w \in \M$.
Fix $z \in X_0$ with $\|z\|=1$.
For each $s \in \R$ with $v+sz\neq 0$,
there exists unique $t(s)>0$ satisfying $t(s)(v+sz)\in \M$.
Then we can easily see
$$t'(0)=-\int_\Omega |v|^\frac{4s}{N-2s}vz\,dx.$$
From
$$
\Ray_0(v+sz)-\Ray_0(v)=\Ray_0(t(s)(v+sz))-\Ray_0(v)
 \geq -\sqrt{\eta}\|t(s)(v+sz)-t(s)v +t(s)v-v\|,
$$
we obtain
$$|\Ray_0'(v)(z)|\leq \sqrt{\eta}\|z+t'(0)v\|\leq
 \sqrt{\eta}(1+1/\sqrt{\S}),$$
which yields $\|\Ray_0'(v)\|\leq \sqrt{\eta}(1+1/\sqrt{\S})$.
\end{proof}

\section{Proof of Theorem~\ref{mainthm} concluded}
\noindent
In the following proof, we will repeatedly use the fact that $\Ray
(\sigma u)=\Ray(u)$ for every $\sigma >0$ and every $u \in \dot
H^s(\R^N) \setminus \{0\}$. We write, for $u \in \dot H^s(\R^N) \setminus \{0\}$,
\[
\Pi (u) = \frac{u}{\|u\|_{L^{2N/(N-2s)}}}.
\]
From Propositions~\ref{one} and \ref{two}, we can find $C_3>0$ with
\[
\Ray(\Pi(\u))
\leq\frac{\Hsnorm{\Uoz}^2+C_1\varepsilon^{N-2s}}
{\left(\int_{\R^N} |\Uoz|^\frac{2N}{N-2s}\,dx-C_2\varepsilon^N
\right)^\frac{N-2s}{N}}
\leq 
\Ray(\Uoz)
+C_3\varepsilon^{N-2s}
\]
for each $\varepsilon \in (0,\onetwe]$, $\delta\in (0,\varepsilon^2]$
and $z \in B_1$.
Hence,
we can find $\bar{\varepsilon}\in (0,\onetwe]$ such that
\[
\Ray(\Pi(\uu))\leq \varpi 2^{2s/N}\S\quad
\text{for each $z \in B_1$},
\]
where $2^{-\frac{2s}{N}}<\varpi <1$.
Now, we fix $\delta\defi\bar{\varepsilon}^2$
and we define a kind of barycenter mapping
\[
\beta(u)\defi\int_{\R^N} 1_{B_K}(x)\, x|u(x)|^{\frac{2N}{N-2s}}\,dx
\quad\text{for each $u\in \dot H^s(\R^N)$ with
$\|u\|_{L^{2N/(N-2s)}}=1$,
}
\]
where $K\defi\sup\{|x| : x\in \Omega\} +1$
and $1_{B_K}$ is the characteristic function for 
$B_K$. We also define 
\[
\bar{c}\defi\inf\left\{
\Ray_0(u)
:
u \in \M, \beta (u)=0
\right\}.
\]
Then, $\bar{c}>\S$.
If not, there is a sequence $\{v_n\}_n\subset\M$
such that $\beta(v_n)=0$ and
$\Ray_0(v_n)\to\S$.
From Lemma~\ref{Ekeland},
we have $\Ray_0'(v_n)\rightarrow 0$.
Then by Proposition~\ref{s-struwe},
taking a subsequence if necessary,
there exist 
$\{\lambda_n\}_n\subset (0,1)$
and $\{z_n\}_n\subset \Omega$
such that $\lambda_n\rightarrow 0$,
$z_n\rightarrow z\in\overline{\Omega}$ and
\[
\text{either \,\, $\Hsnorm{v_{n} - \Pi (U_{\lambda_n,z_n})}=o(1)$
\,\, or \,\,
 $\Hsnorm{v_{n} +\Pi (U_{\lambda_n,z_n})}=o(1)$\,\, as $n\to\infty$}.
\]
From $\beta(v_n)=0$ and $\beta(v_n)\rightarrow z$,
we obtain $0\in \overline{\Omega}$, which is a contradiction.
Now, from Propositions \ref{one} and \ref{two},
we can find a map $f\colon B_1\to \M$ 
which satisfies
\[
\Ray_0(f(z))\leq \varpi 2^{2s/N}\S\quad\text{for each $z\in B_1$,}
\]
\[
\Ray_0(f(z))\leq \frac{\S+\bar{c}}{2}<\bar{c} 
\quad\text{for each $z \in \partial B_1$}
\]
and
\begin{equation}
\label{boundary}
|\beta(f(z))-z|\leq \frac{1}{2}\quad\text{for each $z \in \partial
B_1$}.
\end{equation}
Such $f$ can be obtained by setting
$f(z)\defi u_{\bar{\varepsilon}^2,h_\varepsilon(|z|),z}$
with sufficiently small $\varepsilon>0$,
where
$$h_\varepsilon(t)=
\begin{cases}
\bar{\varepsilon}& \text{for $0\leq t\leq 1/2$,}\\
2(1-t)\bar{\varepsilon}+(2t-1)\varepsilon
& \text{for $1/2\leq t\leq 1$,}
\end{cases}$$
and
we can show \eqref{boundary} 
by a similar argument above which shows $\bar{c}>\S$.
Then, for each $t \in [0,1]$ and $z \in \partial B_1$,
we have $|(1-t)z+t\beta(f(z))|\geq |z|-t|\beta(f(z))-z|\geq 1/2$.
So by using Brouwer's degree theory, we have ${\mathrm{deg}}(\beta\circ f, {\mathrm{Int}}(B_1),0)=1$. Defining
\begin{align*}
c\defi\inf_{g\in G}\max_{x\in B_1}\Ray_0(g(x)),\qquad
G\defi\{g\in C(B_1,\M): \text{$g=f$ on $\partial B_1$ and }
{\mathrm{deg}}(\beta\circ g, {\mathrm{Int}}(B_1),0)=1\},
\end{align*}
we have 
\[
\S<\bar c\leq c\leq \varpi 2^{2s/N}\S.
\]
Now,
we will show there is $u \in \M$ such that
$\nabla\Ray_0(u)=0$ and $\Ray_0(u)=c$.
Assume not.
By Proposition \ref{prop:PS},
we can choose a positive constant $\eta>0$ 
such that
$(\S+c)/2< c-2\eta$, $c+2\eta<\varpi 2^{2s/N}\S$
and $\ZZ(u)\neq 0$ for each $u\in \M$ with $|\Ray_0(u)-c|\leq 3\eta$.
We also choose a locally Lipschitz function $\alpha :\M\rightarrow [0,1]$
such that
$$\alpha(u)=
\begin{cases}
1  & \text{for each $u\in \M$ with $|\Ray_0(u)-c|\leq \eta$,}\\
0  & \text{for each $u\in \M$ with $|\Ray_0(u)-c|\geq 2\eta$.}
 \end{cases}$$
Then we can define $\gamma:[0,1]\times \M\rightarrow \M$
by
\begin{equation}
\label{flow3}
\gamma(0,u)=u\qquad\text{and}\qquad
\frac{d}{dt}\gamma(t,u)=-\frac{2\eta
\alpha(\gamma(t,u))}{\|\ZZ(\gamma(t,u))\|^2}\ZZ(\gamma(t,u));
\end{equation}
see \eqref{flow} and \eqref{flow2}.
Let $g \in G$ such that $\max_{z \in B_1}\Ray_0(g(z))< c+\eta$.
Then we can easily see
$\gamma(t,g(z))=g(z)$ for each $(t,z)\in [0,1]\times \partial B_1$,
which yields ${\mathrm{deg}}(\beta(\gamma(1,g(\cdot))), {\mathrm{Int}}(B_1),0)=1$.
Moreover,
we can find $\Ray_0(\gamma(1,g(z)))\leq c-\eta$ for each $z\in B_1$,
which contradicts the definition of $c$. 
From Proposition~\ref{s-struwe},
we can find that this contradiction proves 
the existence of a nonnegative weak solution to \eqref{problem}. 
By \cite[Theorem~2.5]{IMS},
the obtained solution is positive in $\Omega$.
\qed


\bigskip
\bigskip


\bigskip
\bigskip

\begin{thebibliography}{99}


\bibitem{bahricoron}
{\sc A.\ Bahri, J.-M.\ Coron}, 
On a nonlinear elliptic equation involving the critical Sobolev exponent: the effect of the topology of the domain, 
{\em Comm. Pure Appl. Math.}  {\bf 41} (1988), 253--294.


\bibitem{caffarelli}
{\sc L.\ Caffarelli, L.\ Silvestre}, An extension problem related
to the fractional Laplacian, 
{\em Comm.\ Partial Differential Equations} {\bf 32} (2007), 1245-1260.

\bibitem{capella}
{\sc A.\ Capella},
Solutions of a pure critical exponent problem involving the half-laplacian in annular-shaped domains,
{\em Comm.\ Pure Applied Anal.} {\bf 10} (2011), 1645--1662.

\bibitem{CFY}
\textsc{W.\ Chen, Y.\ Fang, R.\ Yang,}
Semilinear equations involving the fractional Laplacian on domains,
arXiv:1309.7499v1.


\bibitem{classif}
{\sc W.\ Chen, C.\ Li,  B.\ Ou},  
Classification of solutions for an integral equation,
{\em Comm.\ Pure Appl.\ Math.} {\bf 59} (2006), 330--343.

\bibitem{coron}
{\sc J.M.\ Coron}, 
Topologie et cas limite del injections de Sobolev, 
{\em C.R.\ Acad.\ Sc.\ Paris} {\bf 299} (1984), 209--212. 

\bibitem{bestconst}
{\sc A.\ Cotsiolis, N.K.\ Tavoularis}, 
Best constants for Sobolev inequalities	for higher
order fractional derivatives, 
{\em J. Math. Anal. Appl.} {\bf 295} (2004), 225--236.

\bibitem{dancer}
{\sc E.N.\ Dancer}, A note on an equation with critical exponent, 
\emph{Bull.\ Lond.\ Math.\ Soc.} {\bf 20} (1988), 600--602.

\bibitem{dCKP}
\textsc{A.\ Di Castro, T.\ Kuusi, G.\ Palatucci},
Local behavior of fractional $p$-minimizers,
preprint.

\bibitem{DiNezza}
\textsc{E.\ Di Nezza, G.\ Palatucci, E.\ Valdinoci},
Hitchhiker's guide to the fractional Sobolev spaces,
\emph{Bull.\ Sci.\ math.} \textbf{136} (2012), 512--573.


\bibitem{ding} 
{\sc W.Y.\ Ding}, Positive solutions of $\Delta u+u^{(n+2)/(n-2)} = 0$ on contractible domains, {\em J.\ Partial Differ.\ Equ.} {\bf 2} (1989), 83--88.

\bibitem{fallweth}
{\sc M.M.\ Fall, T.\ Weth}, Nonexistence results for a 
class of fractional elliptic boundary value problems, {\em J.\ Funct.\ Anal.}
{\bf 263} (2012), 2205-2227.




\bibitem{IMS}
\textsc{A.\ Iannizzotto, S.\ Mosconi, M.\ Squassina,}
$H^s$ versus $C^0$-weighted minimizers,
preprint.





\bibitem{passaseo} 
{\sc D.\ Passaseo}, 
Multiplicity of positive solutions of nonlinear elliptic equations with critical Sobolev exponent in some contractible domains, 
{\em Manuscripta Math.} {\bf 65} (1989), 147--165.

\bibitem{ros-oton} 
\textsc{X.\ Ros-Oton, J.\ Serra}, 
The Pohozaev identity for the fractional laplacian, 
{\em Arch.\ Rat.\ Mech.\ Anal.}, to appear.


\bibitem{valdserv-tams}
{\sc R.\ Servadei, E.\ Valdinoci},
The Brezis-Nirenberg result for the fractional laplacian, 
{\em Trans.\ Amer.\ Math.\ Soc.} (2014) to appear.

\bibitem{struwe}
{\sc M.\ Struwe}, 
A global compactness result for elliptic boundary 
value problems involving limiting nonlinearities, 
{\em Math.\ Z.} {\bf 187} (1984),  511--517.

\bibitem{willem}
{\sc M.\ Willem}, 
Minimax theorems,
{\em Progress in Nonlinear Differential Equations and their Applications},
{\bf 24},
Birkh\"auser Boston, Inc., Boston, MA, 1996.

\end{thebibliography}
\end{document}